\documentclass[11pt]{article}
\usepackage{amsmath, amscd, amssymb, latexsym, epsfig, color, amsthm}
\usepackage[all]{xy}
\usepackage{subfig}
\usepackage{verbatim}
\usepackage{blkarray}
\numberwithin{equation}{section}

\newtheorem{theorem}{Theorem}[section]
\newtheorem{proposition}[theorem]{Proposition}

\newtheorem{corollary}[theorem]{Corollary}

\newtheorem{lemma}[theorem]{Lemma}

\newtheorem{problem}[theorem]{Problem}

\theoremstyle{definition}

\newtheorem{claim}[theorem]{Claim}

\DeclareMathOperator\lk{\mathrm{lk}}
\DeclareMathOperator\st{\mathrm{st}}

\DeclareMathOperator{\dist}{\mathrm{dist}}

\newcommand{\field}{{\bf k}}

\newcommand{\R}{{\mathbb R}}

\title{The rigidity of the graphs of homology spheres minus one edge}
\author{Hailun Zheng\\
	\small Department of Mathematics \\[-0.8ex]
	\small University of Washington\\[-0.8ex]
	\small Seattle, WA 98195-4350, USA\\[-0.8ex]
	\small \texttt{hailunz@math.washington.edu}
}
\begin{document}
	\maketitle
	\begin{abstract}
		We prove that for any prime homology $(d-1)$-sphere $\Delta$ of dimension $d-1\geq 3$ and any edge $e\in S$, the graph $G(\Delta)-e$ is generically $d$-rigid. This confirms a conjecture of Nevo and Novinsky.
	\end{abstract}
\section{Introduction}
The main object of this paper is the notion of generic rigidity. We now briefly mention a few relevant definitions, defering the rest until  later sections. Recall that a $d$-\emph{embedding} of a graph $G=(V, E)$ is a map $\psi: V\to \mathbb{R}^d$. This embedding is called \emph{rigid} if there exists an $\epsilon>0$ such that if $\psi: V\to \R^d$ satisfies $\dist(\phi(u), \psi(u))<\epsilon$ for every $u\in V$ and $\dist(\psi(u), \psi(v))=\dist(\phi(u), \phi(v))$ for every $\{u,v\}\in E$, then $\dist(\psi(u), \psi(v))=\dist(\phi(u), \phi(v))$ for every $u,v\in V$. A graph $G$ is called \emph{generically $d$-rigid} if the set of rigid $d$-embeddings of $G$ is open and dense in the set of all $d$-embeddings of $G$. 

The first substantial mathematical result concerning rigidity can be dated back to 1813, when Cauchy proved that any bijection between the vertices of two
convex 3-polytopes that induces a combinatorial isomorphism and an isometry of the facets, induces an isometry of the two polytopes. Based on Cauchy's theorem and on later results by Dehn and Alexandrov, in 1975 Gluck \cite{G} gave a complete proof of the fact that the graphs of all simplicial 3-polytopes are generically 3-rigid. Later Whiteley \cite{W} extended this result to the graphs of simplicial $d$-polytopes for any $d\geq 3$. Many other generalizations have been made since, including, for example, the following theorem proved by Fogelsanger.
\begin{theorem}{\rm{\cite{F}}}\label{thm: rigidity}
	Let $d\geq 3$. The graph of a minimal $(d-1)$-cycle complex is generically $d$-rigid. In particular, the graphs of all homology $(d-1)$-spheres are generically $d$-rigid.
\end{theorem}
The rigidity theory of frameworks is a very useful tool for tackling the lower bound conjectures. For a $(d-1)$-dimensional simplicial complex $\Delta$, we define $g_2(\Delta):= f_1(\Delta)-df_0(\Delta)+\binom{d+1}{2}$, where $f_1$ and $f_0$ are the numbers of edges and vertices of $\Delta$, respectively. By interpreting $g_2(\Delta)$ as the dimension of the left kernel of the rigidity matrix of $\Delta$, Kalai \cite{K} proved that the $g_2$-number of an arbitrary triangulated manifold $\Delta$ of dimension at least three is nonnegative (thus reproving the Lower Bound Theorem due to Barnette \cite{B}, \cite{B2}). Furthermore, Kalai showed that $g_2(\Delta)=0$ is attained if and only if $\Delta$ is a stacked sphere. Kalai's theorem was then extended to the class of normal pseudomanifolds by Tay \cite{T}, where Theorem 1.1 served as a key ingredient in the proof. We refer to \cite{KN} for another application of the rigidity theory to the Balanced Lower Bound Theorem.

It might be tempting to conjecture that the graph of a non-stacked homology sphere $\Delta$ minus any edge of $\Delta$ is also generically $d$-rigid. This is not true in general; for example, let $\Delta$ be obtained by stacking over a facet of any $(d-1)$-sphere $\Gamma$, and let $e$ be any edge not in $\Gamma$. In this case the graph of $\Delta-e$ is  not generically $d$-rigid. However, Nevo and Novinsky \cite{NN} showed that this statement does hold if, in addition, one requires that $\Delta$ is \emph{prime} (i.e., $\Delta$ has no missing facets) and $g_2(\Delta)=1$. They raised the following question. 
\begin{problem}{\rm\cite[Problem 2.11]{NN}}
	Let $d\geq 4$ and let $\Delta$ be a prime homology $(d-1)$-sphere. Is it true that for any edge $e$ in $\Delta$, the graph $G(\Delta)-e$ is generically $d$-rigid?
\end{problem} 
In this paper we give an affirmative answer to the above problem. The proof is based on the rigidity theory of frameworks. Specifically, we first verify the base cases $d=4$ and $g_2=1$, and then prove the result by inducting on both the dimension and the value of $g_2$. 

The paper is organized as follows. In Section 2 after reviewing some preliminaries on simplicial complexes, we introduce the rigidity theory of frameworks and summerize several well-known results in this field. We then prove our main result (Theorem \ref{thm: rigidity, minus edge}) in Section 3.
\section{Preliminaries}
A \textit{simplicial complex} $\Delta$ on vertex set $V=V(\Delta)$ is a collection of subsets $\sigma\subseteq V$, called \textit{faces}, that is closed under inclusion, and such that for every $v \in V$, $\{v\} \in \Delta$. The \textit{dimension} of a face $\sigma$ is $\dim(\sigma)=|\sigma|-1$, and the \textit{dimension} of $\Delta$ is $\dim(\Delta) = \max\{\dim(\sigma):\sigma\in \Delta\}$. The \textit{facets} of $\Delta$ are maximal faces of $\Delta$ under inclusion. We say that a simplicial complex $\Delta$ is \textit{pure} if all of its facets have the same dimension. A \textit{missing} face of $\Delta$ is any subset $\sigma$ of $V(\Delta)$ such that $\sigma$ is not a face of $\Delta$ but every proper subset of $\sigma$ is. A pure simplicial complex $\Delta$ is \emph{prime} if it does not have any missing facets. 

For a simplicial complex $\Delta$, we denote the graph of $\Delta$ by $G(\Delta)$. If $G=(V,E)$ is a graph and $U\subseteq V$, then the \emph{restriction} of $G$ to $U$ is the subgraph $G|_U$ whose vertex set is $U$ and whose edge set consists of all of the edges in $E$ that have both endpoints in $U$. We denote by $C(G)$ the graph of the cone over a graph $G$, and by $K(V)$ the complete graph on the vertex set $V$. For brevity of notation, in the following we will use $G + e$ (resp. $G - e$) to denote the graph obtained by adding an edge $e$ to (resp. deleting $e$ from) $G$. 

In this paper we focus on the graphs of a certain class of simplicial complexes. Given an edge $e=\{a,b\}$ of a simplicial complex $\Delta$, the contraction of $e$ to a new vertex $v$ in $\Delta$ is the simplicial complex
\[\Delta^{\downarrow e}:=\{F\in \Delta: a, b\notin F\}\cup \{F\cup\{v\}: F\cap\{a,b\}=\emptyset \;\mathrm{and} \;\mathrm{either}\; F\cup\{a\}\in \Delta \;\mathrm{or}\; F\cup\{b\}\in\Delta\}.\] A simplicial complex $\Delta$ is a \emph{simplicial sphere} if the geometric realization of $\Delta$, denoted as $||\Delta||$, is homeomorphic to a sphere. Let $\tilde{H}_*(\Gamma,\field)$ denote the reduced singular homology of $||\Gamma||$ with coefficients in $\field$. The \textit{link} of a face $\sigma$ is $\lk_\Delta \sigma:=\{\tau-\sigma\in \Delta: \sigma\subseteq \tau\in \Delta\}$, and the \textit{star} of $\sigma$ is $\st_\Delta \sigma:= \{\tau \in\Delta: \sigma\cup\tau\in\Delta \}$. For a pure $(d-1)$-dimensional simplicial complex $\Delta$ and a field $\field$, we say that $\Delta$ is a homology sphere over $\field$ if $\tilde{H}_*(\lk_\Delta \sigma;\field)\cong \tilde{H}_*(\mathbb{S}^{d-1-|\sigma|};\field)$ for every face $\sigma\in\Delta$, including the empty face. We have the following inclusion relations:
\begin{center}
	boundary complexes of simplicial $d$-polytopes $\subseteq$ simplicial $(d-1)$-spheres
	
	\hspace{7.8cm} $\subseteq$ homology $(d-1)$-spheres.
\end{center}
It follows from Steinitz's theorem that when $d=3$, all three classes above coincide. When $d\geq 4$, all three inclusions are strict. 

We are now in a position to review basic definitions of rigidity theory of frameworks. Given a graph $G$ and a $d$-embedding $\phi$ of $G$, we define the matrix $\mathrm{Rig}(G, \phi)$ associated with a graph $G$ as follows: it is an $f_1(G) \times df_0(G)$ matrix with rows labeled by edges of $G$ and columns grouped in blocks of size $d$, with each block labeled by a vertex of $G$; the row corresponding to $\{u,v\}\in E$ contains the vector $\phi(u)-\phi(v)$ in the block of columns corresponding to $u$, the vector $\phi(v)-\phi(u)$ in columns corresponding to $v$, and zeros everywhere else. It is easy to see that for a generic $\phi$ the dimensions of the kernel and image of $\mathrm{Rig}(G, \phi)$ are independent of $\phi$. Hence we define the \emph{rigidity matrix} of $G$ as $\mathrm{Rig}(G, d)=\mathrm{Rig}(G, \phi)$ for a generic $\phi$. It follows from \cite{AR2} that $G$ is generically $d$-rigid if and only if $\mathrm{rank(}\mathrm{Rig}(G, d))=df_0(G)-\binom{d+1}{2}$. The following lemmas summerize a few additional results on framework rigidity.
\begin{lemma}[Cone Lemma, \cite{W}]
	$G$ is generically $(d-1)$-rigid if and only if $C(G)$ is generically $d$-rigid.
\end{lemma}
 Since the star of any face $\sigma$ in a homology sphere is the join of $\sigma$ with the link of $\sigma$, and since the link of $\sigma$ is a homology sphere, Theorem \ref{thm: rigidity} along with the cone lemma implies the following corollary. 
\begin{corollary}\label{cor: star of face}
	Let $d\geq 4$ and let $\Delta$ be a homology $(d-1)$-sphere. Then the graph of $\st_\Delta\sigma$ is generically $d$-rigid for any face $\sigma$ with $|\sigma|\leq d-3$.
\end{corollary}
\begin{lemma}[Gluing Lemma, \cite{AR2}]
	Let $G_1$ and $G_2$ be generically $d$-rigid graphs such that $G_1\cap G_2$ has at least $d$ vertices. Then $G_1\cup G_2$ is also generically $d$-rigid.
\end{lemma}
\begin{lemma}[Replacement Lemma, \cite{K}]
	Let $G$ be a graph and $U$ a subset of $V(G)$. If both $G|_U$ and $G\cup K(U)$ are generically $d$-rigid, then $G$ is generically $d$-rigid.
\end{lemma}
Finally we state a variation of the gluing lemma.
\begin{lemma}\label{lm: replacement lemma}
	Let $G_1$ and $G_2$ be two graphs, and assume that $a, b\in U=V(G_1\cap G_2)$. Assume further that $G_1$ and $G_2$ satisfy the following conditions: 1) the set $U$ contains at least $d$ vertices, including $a$ and $b$, 2) both $G_1$ and $G_2+\{a,b\}$ are generically $d$-rigid, and 3) $G_1|_U= G_2|_U$. Then $G_1\cup G_2$ is also generically $d$-rigid.
\end{lemma}
\begin{proof}
	The second condition implies that $G_1+\{a,b\}$ is generically $d$-rigid. Since $G_i+\{a,b\}$ are generically $d$-rigid for $i=1,2$ and their intersection contains at least $d$ vertices, by the gluing lemma, $G:=(G_1\cup G_2) + \{a,b\}$ is generically $d$-rigid. Note that by condition 3), the restriction of $G$ to $V(G_1)$ is $G_1+\{a,b\}$. Replacing $G_1+ \{a,b\}$ by the generically $d$-rigid graph $G_1$, we obtain the graph $G_1\cup G_2$, which is also generically $d$-rigid by the replacement lemma. 
\end{proof}

\section{Proof of the main theorem}
In this section we will prove our main result, Theorem 3.4. We begin with the following lemma that is originally due to Kalai. We give a proof here for the sake of completeness.
\begin{lemma}\label{lm: edge in a missing face}
	Let $d\geq 4$ and let $\Delta$ be a homology $(d-1)$-sphere. If $\sigma$ is a missing $k$-face in $\Delta$ and $2\leq k\leq d-2$, then $G(\Delta)-e$ is generically $d$-rigid for any edge $e\subseteq \sigma$.
\end{lemma}		
\begin{proof}
	Let $\tau=\sigma\backslash e$. The dimension of $\lk_\Delta \tau$ is
		\[\dim \lk_\Delta \tau=d-1-|\tau|=(d-1)-(|\sigma|-2)\geq d+1-(d-1)=2,\]
	so $\lk_\Delta \tau$ is generically $(d-|\tau|)$-rigid. By Corollary \ref{cor: star of face}, $\st_\Delta \tau$ is generically $d$-rigid. Note that $e\notin \st_\Delta \tau$, and the induced subgraph of $G(\Delta)$ on $W=V(\st_\Delta\tau)$ contains a generically $d$-rigid subgraph $G(\st_\Delta \tau)$. Applying the replacement lemma on $W$ (that is, replacing $G(\Delta)|_{W}$ with $G(\Delta)|_W -e$), we conclude that the resulting graph $G(\Delta)-e$ is also generically $d$-rigid.
\end{proof}
	
The following proposition was mentioned in \cite{NN} without a proof.					
\begin{proposition}\label{lm: edge stress g2=1}
	Let $\Delta$ be a prime homology $(d-1)$-sphere with $g_2(\Delta)=1$, where $d\geq 4$. Then for any edge $e\in \Delta$, the graph $G(\Delta)-e$ is generically $d$-rigid.
\end{proposition}
\begin{proof}
	By Theorem 1.3 in \cite{NN}, $\partial \Delta= \sigma_1 * \sigma_2$, where $\sigma_1$ is the boundary complex of an $i$-simplex for some $i\geq \frac{d+1}{2}$, and $\sigma_2$ is either the boundary complex of a $(d+1-i)$-simplex, or a cycle graph $(c_1,\dots, c_k)$ when $i=d-2$. If $e\in \sigma_1$, then $G(\Delta)-e$ is generically $d$-rigid by Lemma \ref{lm: edge in a missing face}. Now assume that $e$ contains a vertex $v$ in $\sigma_2$. Note that $\sigma_2\backslash \{v\}$ is either a simplex or a path graph. In the former case, the graph of $\Delta\backslash\{v\}$ is the complete graph on $d+1$ vertices, and hence it is generically $d$-rigid. In the latter case, since the graph of $\sigma_1 *\{c_i,c_{i+1}\}$ is also the complete graph on $d+1$ vertices, by the gluing lemma, $G(\Delta \backslash\{v\})$ is generically $d$-rigid. Finally, the graph $G(\Delta)-e$ is obtained by adding to $G(\Delta\backslash\{v\})$ the vertex $v$ and $\deg v-1\geq d$ edges containing $v$. Hence $G(\Delta)-e$ is generically $d$-rigid.
\end{proof}
				
\begin{proposition}\label{lm: edge stress d=4}
	Let $\Delta$ be a prime homology 3-sphere. For any edge $e\in \Delta$, the graph $G(\Delta)-e$ is generically 4-rigid.
\end{proposition}
\begin{proof}
	The proof has a similar flavor to the proof of Proposition 1 in \cite{W2}. If $e$ is an edge in a missing 2-face of $\Delta$, then by Lemma \ref{lm: edge in a missing face}, $G(\Delta)-e$ is generically 4-rigid. Now assume that $e=\{a,b\}$ does not belong to any missing 2-face of $\Delta$. We claim that $\lk_\Delta e=\lk_\Delta a \cap \lk_\Delta b$. If $v\in \lk_\Delta a \cap \lk_\Delta b$, then $e=\{a,b\}$, $\{a,v\}$ and $\{b,v\}$ are edges of $\Delta$. Hence, by our assumption, $\{a,b,v\}\in \Delta$, and so $v\in \lk_\Delta e$. Also if $e'=\{c,d\}\in \lk_\Delta a \cap \lk_\Delta b$, then $e'*\partial e\subseteq \Delta$. Since $e$ does not belong to any missing 2-face of $\Delta$, it follows that $c, d \in \lk_\Delta e$. Hence $(e'*\partial e)\cup(e*\partial e')\subseteq \Delta$, which by the primeness of $\Delta$ implies that $e*e'\subseteq \Delta$, i.e., $e'\in \lk_\Delta e$. Finally, if $\lk_\Delta a\cap \lk_\Delta b$ contains a 2-dimensional face $\tau$ whose boundary edges are $e_1,e_2$ and $e_3$, then the above argument implies that $e_i\cup\{b\}\in \lk_\Delta a$ for $i=1,2,3$.  Hence $\partial (\tau\cup\{b\})\subseteq\lk_\Delta a$, and so $\lk_\Delta a$ is the boundary complex of a 3-simplex. This contradicts the fact that $\Delta$ is prime. We conclude that both $\lk_\Delta e$ and $\lk_\Delta a\cap \lk_\Delta b$ are 1-dimensional. Furthermore, $\lk_\Delta a\cap \lk_\Delta b\subseteq \lk_\Delta e$. However, it is obvious that the reverse inclusion also holds. This proves the claim.
			
	If $\lk_\Delta e$ is a 3-cycle, then the filled-in triangle $\tau$ determined by $\lk_\Delta e$ is not a face of $\Delta$. Otherwise, by the fact that $\tau\cup (\lk_\Delta e *\{a\})$ and $\tau\cup (\lk_\Delta e*\{b\})$ are subcomplexes of $\Delta$ and by the primeness of $\Delta$, we obtain that $\tau\cup\{a\},\tau\cup\{b\}\in \Delta$. Then since $\lk_\Delta e=\lk_\Delta a\cap \lk_\Delta b$, we conclude that $\tau\in \lk_\Delta e$, contradicting that $\lk_\Delta e$ is 1-dimensional. Hence we are able to construct a new sphere $\Gamma$ from $\Delta$ by replacing $\st_\Delta e$ with the suspension of $\tau$ (indeed, $\Gamma$ and $\Delta$ differ in a bistellar flip), and therefore $G(\Delta)-e=G(\Gamma)$ is generically 4-rigid. Next we assume that $\lk_\Delta e$ has at least 4 vertices. By \cite[Proposition 2.3]{NN}, the edge contraction $\Delta^{\downarrow e}$ of $\Delta$ is also a homology sphere. Assume that in a $d$-embedding $\psi$ of $G(\Delta)$, both $a$ and $b$ are placed at the origin, $V(\lk_\Delta e)=\{u_1,\dots, u_l\}$, $V(\lk_\Delta a) -V(\lk_\Delta e)=\{v_1,\dots, v_m\}$ and $V(\lk_\Delta b) -V(\lk_\Delta e)=\{w_1,\dots, w_n\}$. The rigidity matrix of $G(\Delta)-e$ can be written as a block matrix
			\[
				M:=\mathrm{Rig}(G(\Delta)-e, \psi)=
				\begin{pmatrix}
				A & B\\
				0 & R
				\end{pmatrix},
		    \]	
	where the columns of $B$ and $R$ correspond to the vertices in $\st_\Delta a\cup\st_\Delta b$, and the rows of $R$ correspond to the edges containing either $a$ or $b$ but not both. For convenience, we write $\mathbf{v_i}$ (resp. $\mathbf{u_i}$, $\mathbf{w_i}$) to represent $\psi(v_i)$ (resp. $\psi(u_i)$ and $\psi(w_i)$). Then
		    \[
		    R=\begin{blockarray}{cccccccccccc}
		    v_1 & \dots & v_m & u_1 & \dots & u_\ell & w_1 & \dots & w_n & a & b\\
		    \begin{block}{(ccccccccccc)c}
		    \bf{v_1} & & & & & & & & &-\bf{v_1} & \\
		    & \ddots & & & & & & & &\vdots & \\
		    & & \bf{v_m} & & & & & & &-\bf{v_m} & \\
		    & & & \bf{u_1} & & & & & &-\bf{u_1} & &(*_1)\\
		    & & & & \ddots & & & & & \vdots & &\vdots\\
		    & & & & & \bf{u_\ell} & & & &-\bf{u_\ell} & &(*_\ell)\\
		    & & & & & & \bf{w_1} & & & &-\bf{w_1} \\
		    & & & & & & & \ddots & & &\vdots \\
		    & & & & & & & & \bf{w_n} & &-\bf{w_n} \\
		    & & & \bf{u_1} & & & & & & &-\bf{u_1} &({**}_1)\\
		    & & & & \ddots & & & & & & \vdots &\vdots\\
		    & & & & & \bf{u_\ell} & & & & &-\bf{u_\ell} &({**}_\ell)\\
		    \end{block}		
		    \end{blockarray},
		    \]
	where the rest of the entries not indicated above are 0.
	We apply the following row and column operations to matrix $M$: first add the last four columns, i.e. columns corresponding to $b$ to the corresponding columns of $a$, then substract row $(*_i)$ from the row $({**}_i)$ for $i=1,\dots,\ell$. This gives
		    \[
		    M'(\psi')=
		    \begin{pmatrix}
		    \mathrm{Rig}(G(\Delta^{\downarrow e}), \psi') & * \\
		     0 & -\bf{u_1} \\
		     \vdots & \vdots \\
		     0 & -\bf{u_l}\\		
		    \end{pmatrix},
		    \]
    where $\psi'$ is the $4$-embedding of $G(\Delta^{\downarrow e})$ induced by $\psi$, where $\psi'(v)=\psi(a)=\psi(b)$ for the new vertex $v$, and $\psi'(x)=\psi(x)$ for all other vertices $x\neq a,b$. Since $\ell=|V(\lk_\Delta e)|\geq 4$, it follows that the last four columns of $M'(\psi')$ are linearly independent. Hence for a generic $\psi'$, \[\mathrm{rank}(M)=\mathrm{rank}(M'(\psi'))=\mathrm{rank}({\mathrm{Rig}(G(\Delta^{\downarrow e}), 4)})+4=(4f_0(\Delta^{\downarrow e})-10)+4=4f_0(\Delta)-10.\]
    Since $4f_0(\Delta)-10$ is the maximal rank that the rigidity matrix of a 4-dimensional framework with $f_0(\Delta)$ vertices can have, and a small generic perturbation of $a$ and $b$ preserves the rank of the rigidity matrix, we conclude that $\mathrm{rank}(\mathrm{Rig}(G(\Delta)-e, 4))=\mathrm{rank}(M)=4f_0(\Delta)-10$. Hence $G(\Delta)-e$ is generically 4-rigid.
\end{proof}
		
	In the following we generalize the previous proposition to the case of $d>4$ by inducting on the dimension and the value of $g_2$. We fix some notation here. If a homology $(d-1)$-sphere $\Delta$ is the connected sum of $n$ prime homology spheres $S_1,\cdots, S_n$, then each $S_i$ is called a \emph{prime factor} of $\Delta$. In particular,  $\Delta$ is called \emph{stacked} if each $S_i$ is the boundary complex of a $d$-simplex. For every stacked $(d-1)$-sphere $\Delta$ with $d\geq 3$, there exists a unique simplicial $d$-ball with the same vertex set as $\Delta$ and whose boundary complex is $\Delta$; we denote it by $\Delta(1)$. We refer to such a ball as a \emph{stacked ball}.
\begin{theorem}\label{thm: rigidity, minus edge}
	Let $d\geq 4$ and let $\Delta$ be a prime homology $(d-1)$-sphere with $g_2(\Delta)>0$. Then for any edge $e\in G(\Delta)$, the graph $G(\Delta)-e$ is generically $d$-rigid.
\end{theorem}
\begin{proof}
	The two base cases $g_2(\Delta)=1, d\geq 4$ and $d=4, g_2(\Delta)\geq 1$ are proved in Proposition \ref{lm: edge stress g2=1} and \ref{lm: edge stress d=4} respectively. Now we assume that the statement is true for every prime homology $(d_0-1)$-sphere $S$ with $5\leq d_0\leq d$ and $1\leq g_2(S)< g_2(\Delta)$ and every edge $e\in S$. The result follows from the following two claims. 
\end{proof}
\begin{claim}
	Under the above assumptions, if, furthermore, $g_2(\lk_\Delta u)=0$ for some vertex $u\in V(\Delta)$, then $G(\Delta)-e$ is generically $d$-rigid for any edge $e\in \Delta$.
\end{claim}
\begin{proof}
	Since $\lk_\Delta u$ is at least 3-dimensional and since $g_2(\lk_\Delta u)=0$, it follows that $\lk_\Delta u$ is a stacked sphere. Also since $\Delta$ is prime, the interior faces of the stacked ball $(\lk_\Delta u)(1)$ are not faces of $\Delta$ (or otherwise such a face together with $u$ will form a missing facet of $\Delta$). Let \[\Gamma:=(\Delta\backslash\{u\})\cup (\lk_\Delta u)(1).\] Then $\Gamma$ is a homology $(d-1)$-sphere but not necessarily prime. (For more details on this and similar operations, see \cite{Z}.)  Also by the primeness of $\Delta$, every missing facet $\sigma$ of $\Gamma$ must contain a missing facet of $\lk_\Delta u$. Pick a missing facet $\tau$ of $\lk_\Delta u$ and assume that there are $k$ prime factors of $\Gamma$ that contain $\tau$. We first find two facets of $(\lk_\Delta u)(1)$ that contain $\tau$ and say they are $\{v_0\}\cup\tau$ and $\{v_k\}\cup\tau$. Now assume that the $k$ prime factors of $\Gamma$ are $S_{1}, S_{2},\dots, S_{k}$, and each of them satisfies $S_i\cap S_{i+1}=\tau\cup\{v_i\}$ for some other vertices $v_1,\dots, v_{k-1}\in \Delta $ and $1\leq i\leq k-1$. Furthermore, $S_j\cap (\lk_\Delta u)(1)=\tau\cup\{v_j\}$ for $j=0, k$. Let $G_\tau:=E \cup(\cup_{i=1}^{k} G(S_i))$, where $E$ is the set of edges connecting $u$ and the vertices in $\tau\cup \{v_0,v_k\}$. Since an arbitrary edge $e$ of $G(\Delta)$ either contains $u$ or belongs to one of $S_i$'s, it follows that $G(\Delta)-e=\cup (G_{\tau}-e)$, where the union is taken over all missing facets $\tau$ of $\lk_\Delta u$. By the gluing lemma, it suffices to show that $G_\tau -e$ is generically $d$-rigid for any $\tau$ and edge $e\in G_\tau$. We consider the following two cases:
\begin{figure}[h]
			\centering
			\subfloat[The subgraph $G_\tau$: fixing a missing facet $\tau$ of $\lk_\Delta u$, there are three corresponding prime factors $S_1, S_2,S_3$. Here $S_2$ is the boundary complex of the 3-simplex.]{\includegraphics[scale=0.7]{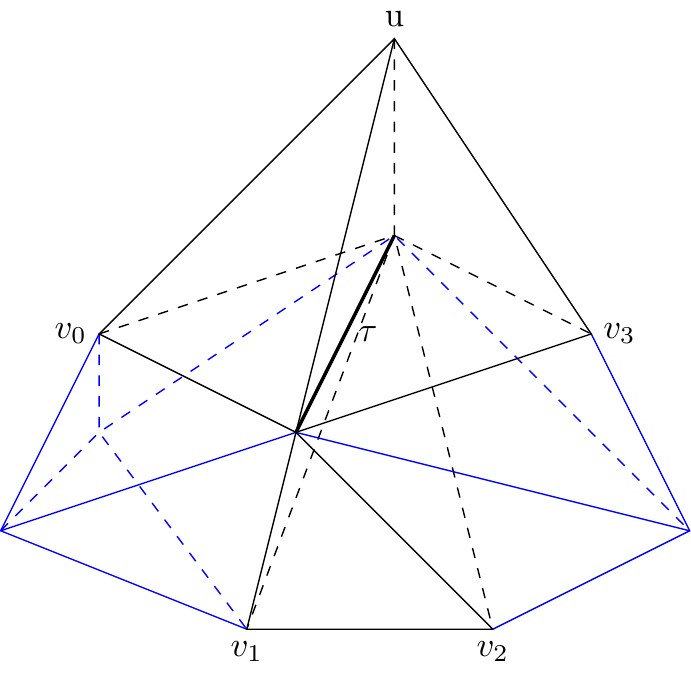}}
			\hspace{20mm}
			\subfloat[$G_\tau'$, obtained from $G_\tau$ by replacing all the blue edges in $G_\tau$ with the red edges $\{v_0,v_1\}$ and $\{v_2,v_3\}$. ]{\includegraphics[scale=0.7]{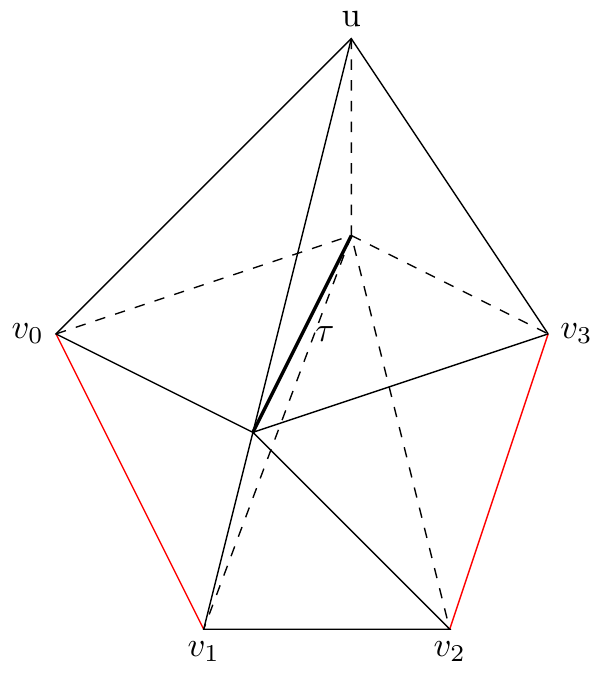}}
			\caption{The corresponding graphs $G_\tau$ and $G_\tau'$, given the graph $G=G(\Delta)$ and a missing facet in a vertex link.}
\end{figure}
		
	{\bf{Case 1}}: $e\in S_i$ for some $i$, $S_i$ is not the boundary complex of the $d$-simplex, and $e\notin S_j$ for any other $j\neq i$. Since $G(S_i)$ is a generically $d$-rigid subgraph of $\Gamma$, it follows that \[g_2(S_i)\leq g_2(\Gamma)=g_2(\Delta)-f_0(\lk_\Delta u)+d<g_2(\Delta).\] Furthemore, by the inductive hypothesis on $g_2$, $G(S_i)-e$ is generically $d$-rigid for any edge $e\in \Delta$. Also since $G(S_i)$ is the induced subgraph of $G_\tau$ on $V(S_i)$, by the replacement lemma, $G_\tau -e$ is generically $d$-rigid.
		
	{\bf{Case 2}}: either $e\in S_i$ for some $i$ and $S_i$ is the boundary complex of the $d$-simplex (in this case the edge $\{v_{i-1}, v_i\}\in S_i$), or $e\in \lk_\Delta u$, or $u\in e$. Hence $e\in G_{\tau}':=G(\tau*C)$, where $C$ is the cycle graph $(u, v_0,\dots, v_k)$. By Lemma \ref{lm: edge stress g2=1}, $G_{\tau}' -e$ is generically $d$-rigid for any edge $e$. The graph $G_\tau-e$ can be recovered from $G_{\tau}'-e$ by replacing each edge $\{v_{i-1},v_i\}$ with the edges in $S_i\backslash G_{\tau}'$ whenever $S_i$ is not the boundary complex of the $d$-simplex. Note that nothing needs to be done when $S_i$ is the boundary complex of a simplex, since $S_i$ is already a subcomplex of $G_\tau'$. (See Figure 1 for an illustration in a lower dimension case.) Repeatedly applying Lemma \ref{lm: replacement lemma} with $\{a,b\}=\{v_{i-1}, v_i\}$, $G_1= G(S_i)-e$ and $G_2+\{a,b\}=G_{\tau}' -e$, we conclude that $G_\tau -e$ is also generically $d$-rigid.
\end{proof}
\begin{claim}
	Under the above assumption, if, furthermore, every vertex link of $\Delta$ has $g_2\geq 1$, then $G(\Delta)-e$ is generically $d$-rigid for any edge $e\in \Delta$. 
\end{claim}
\begin{proof}
	Assume that there is a vertex $u\in \Delta$ such that $\lk_\Delta u$ is the connected sum of prime factors $S_1,\dots,S_k$ and $e=\{v,w\}\in S_i$ for some $i$. If $e$ is an edge in a missing facet of $\lk_\Delta u$ (which is also a missing $(d-2)$-face of $\Delta$), then by Lemma \ref{lm: edge in a missing face}, $G(\Delta)-e$ is generically $d$-rigid. 
		
	Otherwise, assume first that $g_2(S_i)\neq 0$. Then $G(S_i)-e$ is generically $(d-1)$-rigid by the inductive hypothesis on the dimension. Hence by the gluing lemma and cone lemma, we obtain that $G(\st_\Delta u)-e$ is generically $d$-rigid. By the replacement lemma, $G(\Delta)-e$ is generically $d$-rigid.
		
	Finally, assume that $S_i$ is the boundary complex of a $(d-1)$-simplex, or equivalently, $\lk_\Delta \{u,v,w\}$ is the boundary complex of a $(d-3)$-simplex. If furthermore for any vertex $x\in \lk_\Delta e$, the link of $\{x,v,w\}$ in $\Delta$ is the boundary complex of $(d-3)$-simplex, then $\lk_\Delta e$ must be the boundary complex of a $(d-2)$-simplex. Hence $\lk_\Delta v$ is obtained by adding a pyramid over a facet $\sigma$ of some $(d-2)$-sphere, and $w$ is the apex of the pyramid. Now we construct a new homology $(d-1)$-sphere $\Delta'$ as follows: first delete the edge $e$ from $\Delta$, then add the faces $\sigma$, $\sigma\cup\{v\}$ and $\sigma\cup\{w\}$ to $\Delta$. It follows that $G(\Delta')=G(\Delta)-e$, which implies that $G(\Delta)-e$ is generically $d$-rigid. 
		
	Otherwise, there exists a vertex $x$ such that $\lk_\Delta \{x,v,w\}$ is not the boundary complex of $(d-3)$-simplex. Then we may show that $G(\Delta)-e$ is generically $d$-rigid by applying the same argument as above on $\lk_\Delta x$. This proves the claim.
\end{proof}
\section*{Acknowledgements}
	The author was partially supported by a graduate fellowship from NSF grant DMS-1361423. I thank Isabella Novik for helpful comments and discussions.
	\bibliographystyle{amsplain}
	
	\end{document}